\documentclass[a4paper,11pt]{article}
\usepackage[latin1]{inputenc}
\usepackage[english]{babel}
\usepackage{amsmath}
\usepackage{amsfonts}
\usepackage{amssymb}
\usepackage{epsfig}
\usepackage{amsopn}
\usepackage{amsthm}
\usepackage{color}
\usepackage{graphicx}
\usepackage{enumerate}
\newtheorem{theorem}{Theorem}[section]
\newtheorem{corollary}[theorem]{Corollary}
\newtheorem{lemma}[theorem]{Lemma}

\newtheorem{definition}[theorem]{Definition}

\newtheorem*{theorem*}{Theorem}
\newtheorem*{lemma*}{Lemma}
\newtheorem*{remark*}{Remark}
\newtheorem*{definition*}{Definition}
\newtheorem*{proposition*}{Proposition}
\newtheorem*{corollary*}{Corollary}
\numberwithin{equation}{section}
\newcommand{\real}{\mathbb{R}}

\let\ced=\c         % cedilla
\def\qed{\,\unskip\kern 6pt \penalty 500
\raise -2pt\hbox{\vrule \vbox to8pt{\hrule width 6pt
\vfill\hrule}\vrule}\par}
\definecolor{darkblue}{rgb}{0.05, .05, .65}
\definecolor{darkgreen}{rgb}{0.1, .65, .1}
\definecolor{darkred}{rgb}{0.8,0,0}
\newcommand{\beqn}{\begin{equation}}
\newcommand{\eeqn}{\end{equation}}
\newcommand{\bear}{\begin{eqnarray}}
\newcommand{\eear}{\end{eqnarray}}
\newcommand{\bean}{\begin{eqnarray*}}
\newcommand{\eean}{\end{eqnarray*}}
\begin{document}

\title{\huge \bf Extinction for a singular diffusion equation with strong gradient absorption revisited}

\author{
\Large Razvan Gabriel Iagar\,\footnote{Instituto de Ciencias
Matem\'aticas (ICMAT), Nicol\'as Cabrera 13-15, Campus de
Cantoblanco, 28049, Madrid, Spain, \textit{e-mail:}
razvan.iagar@icmat.es},\footnote{Institute of Mathematics of the
Romanian Academy, P.O. Box 1-764, RO-014700, Bucharest, Romania.}
\\[4pt] \Large Philippe Lauren\ced{c}ot\,\footnote{Institut de
Math\'ematiques de Toulouse, CNRS UMR~5219, Universit\'e de
Toulouse, F--31062 Toulouse Cedex 9, France. \textit{e-mail:}
Philippe.Laurencot@math.univ-toulouse.fr}\\ [4pt] }
\date{\today}
\maketitle

%%%%%%%%%%%%%%%%%%%%%%%%%%%%%%%%%%%%%%%%%%%%%%%%%%%%%%%
\begin{abstract}
When $2N/(N+1)<p<2$ and $0<q<p/2$, non-negative solutions to the singular diffusion equation with gradient absorption 
$$
\partial_tu-\Delta_p u + |\nabla u|^q=0 \ \text{ in }\ (0,\infty)\times\mathbb{R}^N
$$ 
vanish after a finite time. This phenomenon is usually referred to as finite time extinction and takes place provided the initial condition $u_0$ decays sufficiently rapidly as $|x|\to\infty$. On the one hand, the optimal decay of $u_0$ at infinity guaranteeing the occurence of finite time extinction is identified. On the other hand, assuming further that $p-1<q<p/2$, optimal extinction rates near the extinction time are derived.
\end{abstract}

%%%%%%%%%%%%%%%%%%%%%%%%%%%%%%%%%%%%%%%%%%%%%%%%%%%%%%%

\vspace{2.0 cm}

%%%%%%%%%%%%%%%%%%%%%%%%%%%%%%%%%%%%%%%%%%%%%%%%%%%%%%%
%%%%%%%%%%%%%%%%%%%%%%%%%%%%%%%%%%%%%%%%%%%%%%%%%%%%%%%

\medskip

%%%%%%%%%%%%%%%%%%%%%%%%%%%%%%%%%%%%%%%%%%%%%%%%%%%%%%%

%%%%%%%%%%%%%%%%%%%%%%%%%%%%%%%%%%%%%%%%
%%%%%%%%%%%%%%%%%%%%%%%%%%%%%%%%%%%%%%%%
\section{Introduction}
%%%%%%%%%%%%%%%%%%%%%%%%%%%%%%%%%%%%%%%%
%%%%%%%%%%%%%%%%%%%%%%%%%%%%%%%%%%%%%%%%

We study some properties related to the phenomenon of finite time extinction of non-negative solutions to the initial value problem in $\real^N$ for the singular diffusion equation with gradient absorption
\begin{eqnarray}
% \nonumber % Remove numbering (before each equation)
  \partial_tu-\Delta_pu+|\nabla u|^q=0, && (t,x)\in(0,\infty)\times\real^N,\label{eq1} \\
  u(0)=u_0, && x\in\real^N,\label{IC}
\end{eqnarray}
when the exponents $p$ and $q$ satisfy
\begin{equation}\label{exp}
p_c:=\frac{2N}{N+1}<p<2, \qquad 0<q<\frac{p}{2},
\end{equation}
the $p$-Laplacian operator being given as usual by
$$
\Delta_p u(t,x)=\mathrm{div}(|\nabla u|^{p-2}\nabla u)(t,x), \qquad (t,x)\in(0,\infty)\times\real^N.
$$
We also assume throughout the paper that the initial condition $u_0$ enjoys the following properties:
\begin{equation}\label{regIC}
u_0\in L^1(\real^N)\cap W^{1,\infty}(\real^N), \qquad u_0(x)\ge 0, \ x\in\real^N, \qquad u_0\not\equiv0.
\end{equation}
According to the analysis performed in \cite[Section~6]{IL12}, the Cauchy problem \eqref{eq1}-\eqref{IC}, with initial condition satisfying \eqref{regIC}, has a unique non-negative (viscosity) solution $u$, the notion of viscosity solutions being the one developed in \cite{OS} to handle the singularity of the diffusion, see \cite[Definition~6.1]{IL12}. It is also a weak solution by \cite[Theorem~6.2]{IL12}. Moreover, in the range of exponents \eqref{exp}, the phenomenon of extinction of the solution $u$ in finite time occurs according to \cite[Theorem~1.2(iii)]{IL12} provided that the initial condition $u_0$ decays sufficiently fast as $|x|\to\infty$. More precisely, it is shown that, if
\begin{equation}\label{tail.fast}
u_0(x)\leq C_0|x|^{-(p-Q)/(Q-p+1)}, \qquad x\in\real^N,
\end{equation}
for some $C_0>0$ and suitable $Q>0$ (which is equal to $q$ if $q>q_1:=\max\{p-1,N/(N+1)\}$ and is arbitrary in $(q_1,p/2)$ otherwise), then
\begin{equation}
T_e := \sup\{ t\ge 0\ :\ u(t)\not\equiv 0\} \label{PhTe}
\end{equation}
is finite and positive. More recent works such as \cite{IL17, ILS17} go further in characterizing how the finite time extinction takes place, showing that (under suitable conditions on $u_0$) an even more striking phenomenon, the \emph{instantaneous shrinking of the support} takes place for $q\in (0,p-1)$. More precisely, the positivity set $\mathcal{P}(t)$ of $u$ at time $t$ defined by
\begin{equation}
\mathcal{P}(t)=\{x\in\real^N: u(t,x)>0\} \label{PhPosSet}
\end{equation}
is compact (and localized uniformly in $t$) for any $t\in(0,T_e)$, even if $u_0(x)>0$ for any $x\in\real^N$. In \cite{IL17}, Eq. \eqref{eq1} with critical exponent $q=p-1$ is studied thoroughly and both optimal extinction rates and precise extinction profiles (in separate variable form) are given, provided that the initial condition $u_0$ is radially symmetric, radially non-increasing in $|x|$, and has an exponential spatial tail as $|x|\to\infty$. As a by-product, it is also shown that \emph{simultaneous extinction} occurs both for $q=p-1$ and for $q\in(p-1,p/2)$, that is,
$$
u(t,x)>0, \qquad {\rm for \ any } \ (t,x)\in(0,T_e)\times\real^N.
$$
However, it was noticed already in \cite[Theorem~1.1 and Theorem~1.2]{ILS17} for the range $0<q<p-1$ that the previous tail \eqref{tail.fast} is not optimal for finite time extinction to take place and our first main result is devoted to the identification of the optimal decay of $u_0$ as $|x|\to\infty$ guaranteeing the occurrence of this phenomenon.

\begin{theorem}[Optimal tail for extinction]\label{th.tail}
Let $u$ be a solution to the Cauchy problem \eqref{eq1}-\eqref{IC} with exponents satisfying \eqref{exp} and an initial condition $u_0$ satisfying \eqref{regIC}.

\begin{itemize}
\item[(a)] Assume further that
\begin{equation}\label{tail.optimal}
u_0(x)\leq C_0(1+|x|)^{-q/(1-q)}, \qquad x\in\real^N,
\end{equation}
for some $C_0>0$. Then the extinction time $T_e$ of $u$ defined in \eqref{PhTe} is positive and finite.

\item[(b)] If
\begin{equation}\label{tail.slow}
\lim\limits_{|x|\to\infty}|x|^{q/(1-q)}u_0(x)=\infty, \qquad u_0(x)>0 \ \ \text{ for  any} \ x\in\real^N,
\end{equation}
then $T_e=\infty$ and $\mathcal{P}(t)=\real^N$ for any $t>0$.
\end{itemize}
\end{theorem}

An obvious consequence of Theorem~\ref{th.tail} is the optimality of the tail behavior \eqref{tail.optimal} for finite time extinction to occur. Furthermore, it strictly improves \cite[Theorem~1.2(iii)]{IL12}. Indeed, since $p>q+Q$, the exponent $Q$ being introduced in \eqref{tail.fast}, it follows that $(p-Q)(1-q)>q(Q-p+1)$ or equivalently
$$
\frac{p-Q}{Q-p+1}>\frac{q}{1-q}.
$$
Consequently, the decay assumed in \eqref{tail.fast} is strictly faster than the optimal one \eqref{tail.optimal}. Let us also remark that we state Theorem~\ref{th.tail} here for exponents $p$ and $q$ satisfying \eqref{exp}, but for the range of exponents $0<q<p-1$, it is already proved in \cite[Theorems~1.2 and~1.3]{ILS17}.

Once it is known that finite time extinction takes place, a further important step in understanding the extinction mechanism is to identify the behavior of the solution $u$ to the Cauchy problem \eqref{eq1}-\eqref{IC} as $t\to T_e$, where $T_e$ is the extinction time defined in \eqref{PhTe}. To this end, a first point is to determine the extinction rate, that is, the precise (optimal) space and time scales in which $u(t)$ vanishes as $t\to T_e$. This is the second main result of the present note. Before stating it, let us introduce the exponents
\begin{equation}\label{exp.alpha}
\alpha:=\frac{p-q}{p-2q}, \qquad \beta := \frac{q-p+1}{p-2q},
\end{equation}
which will be used throughout the paper.

\begin{theorem}[Optimal extinction rate]\label{th.rate}
Assume that $p\in (p_c,2)$ and $p-1<q<p/2$. Let $u$ be the solution to the Cauchy problem \eqref{eq1}-\eqref{IC} with an initial condition $u_0$ satisfying \eqref{regIC} as well as the decay property
\begin{equation}\label{tail.fast1}
0\leq u_0(x)\leq K_0|x|^{-(p-q)/(q-p+1)}, \qquad x\in\real^N,
\end{equation}
for some $K_0>0$. Then there exist two positive constants $c_\infty$ and $C_\infty$ (depending on $N$, $p$, $q$, and the initial condition), such that
\begin{equation}\label{rate.optimal}
c_\infty(T_e-t)^{\alpha}\leq\|u(t)\|_{\infty}\leq C_\infty(T_e-t)^{\alpha}, \qquad t\in (T_e/2,T_e).
\end{equation}
Furthermore, there are two positive constants $c_1$ and $C_1$ (depending on $N$, $p$, $q$, and the initial condition), such that
\begin{equation}\label{rate.optimal.L1}
c_1 (T_e-t)^{\alpha-N\beta}\leq\|u(t)\|_{1}\leq C_1(T_e-t)^{\alpha-N\beta}, \qquad t\in (T_e/2,T_e).
\end{equation}
\end{theorem}

The proof of these optimal bounds near extinction is very clear-cut, elementary and based on a rather simple \emph{energy technique}, and its application is thus likely to extend beyond \eqref{eq1}. For instance, we refer the interested reader to the companion paper \cite{ILFDE} where a related approach allows us to derive optimal extinction rates for a fast diffusion equation with zero order strong absorption.

Let us point out here that the range of application of Theorem~\ref{th.rate} is narrower than that of Theorem~\ref{th.tail}, as we have to impose two further restrictions. The first one is related to the decay at infinity of the initial condition $u_0$, which is required to be much faster than the optimal one \eqref{tail.optimal} identified in Theorem~\ref{th.tail}. As a consequence, we do not know whether, for initial conditions satisfying \eqref{tail.optimal} but not \eqref{tail.fast1}, the outcome of Theorem~\ref{th.rate} remains valid. The second restriction is related to the range of the exponent $q$ which is restricted to the smaller interval $(p-1,p/2)$ in Theorem~\ref{th.rate}. This assumption is seemingly only technical and some arguments in that direction are the following: on the one hand, for the critical case $q=p-1$, the extinction rate \eqref{rate.optimal} is already proved in \cite{IL17} for radially symmetric initial data, though by a completely different technique. In addition, an optimal upper bound near the extinction is derived for the $L^2$-norm of $u$. On the other hand, when $q\in (0,p-1)$, the behavior near the extinction time is studied in \cite[Proposition~5.1]{ILS17}. Although we show the validity of the lower bound in \eqref{tail.optimal} in that case as well, we are unfortunately only able to obtain upper bounds of the form $C(\varepsilon) (T_e-t)^{\alpha-\varepsilon}$ without a suitable control on the behavior of $C(\varepsilon)$ as $\varepsilon\to 0$. A proof of the upper bound in \eqref{tail.optimal} when $q\in (0,p-1)$ might however require a different approach. Indeed, in that case, as we previously mentioned, \emph{instantaneous shrinking} takes place, that is, the support of $u(t)$ is compact for all $t\in (0,T_e)$, and identifying the optimal rate of shrinking of the support might be an helpful piece of information.

We finally mention that optimal extinction rates have also been studied for the related fast diffusion equation with zero order strong absorption
\begin{equation}\label{FDE}
\partial_tu-\Delta u^m+u^q=0, \qquad (t,x)\in(0,\infty)\times\real^N,
\end{equation}
for exponents $m\in((N-2)_+/N,1]$ and $q\in(0,1)$ but, unlike the present contribution, many works focus on the one-dimensional case $N=1$ \cite{CMM95, FV, GHV92, GV94a, HV92}. Extinction rates in the general $N$-dimensional case are only studied in \cite{FH87} for $m=1$ and a restricted class of initial conditions and in the companion paper \cite{ILFDE} for $m\in ((N-2)_+/N,1)$ and $q\in (m,1)$.

\section{Optimal tail for extinction}\label{sec.tail}

In this section we prove Theorem~\ref{th.tail}. The technique of the proof is based on constructing suitable supersolutions with finite time extinction, on the one hand, and subsolutions which are positive everywhere, on the other hand. We thus need two preparatory, technical lemmas. As already explained in the Introduction, Theorem~\ref{th.tail} is already proved in \cite{ILS17} in the range $0<q<p-1$, so that the novelty of this section is the fact that we handle the case $q\in[p-1,p/2)$.

\subsection{Notions of subsolution and supersolution}

We recall here for the sake of completeness (according to \cite[Definition~6.1]{IL12}) the notions of subsolution and supersolution that we use in the sequel. They are to be understood in the \emph{viscosity sense} and follow the general (abstract) approach developed in \cite{IS, OS}, where the class of admissible functions for comparison is reduced in order to cope well with the singular diffusion featured in \eqref{eq1}. In order to introduce the class of admissible functions for comparison, let $\mathcal{F}_p$ be the set of functions $\xi\in C^{2}([0,\infty))$ such that
\begin{equation}
\xi(0)=\xi'(0)=\xi''(0)=0, \quad \xi''(r)>0 \ \ {\rm for \ all} \ r>0, \quad \lim\limits_{r\to 0}|\xi'(r)|^{p-2}\xi''(r)=0. \label{PhF1}
\end{equation}
Notice that l'Hospital's rule and \eqref{PhF1} entail that
\begin{equation}
\lim\limits_{r\to 0} \frac{|\xi'(r)|^{p-1}}{r} = 0. \label{PhF2}
\end{equation}
As a simple example of a function in the class $\mathcal{F}_p$, any power $\xi(r)=r^{\sigma}$ can be taken, provided $\sigma>p/(p-1)$. We next define the class $\mathcal{A}$ of admissible comparison functions. A function $\psi\in C^{2}((0,\infty)\times\real^N)$ belongs to $\mathcal{A}$ if, for any point $(t_0,x_0)\in(0,\infty)\times\real^N$ such that $\nabla\psi(t_0,x_0)=0$, there exist $\delta>0$, a function $\xi\in\mathcal{F}_p$ and a modulus of continuity $\omega\in C([0,\infty))$ with $\omega(t)/t\to 0$ as $t\to 0$ enjoying the following property: for any $(t,x)\in (t_0-\delta,t_0+\delta)\times B_\delta(x_0)$, there holds:
\begin{equation}\label{visc}
|\psi(t,x)-\psi(t_0,x_0)-\partial_t\psi(t_0,x_0)(t-t_0)|\leq \xi(|x-x_0|)+\omega(|t-t_0|).
\end{equation}
With this construction, we now define viscosity subsolutions and supersolutions.
\begin{definition}\label{def.visc}
Let $T>0$.
\begin{itemize}
\item[(a)] An upper semicontinuous function $u:(0,T)\times\real^N\mapsto\real$ is a \emph{viscosity subsolution} to \eqref{eq1} if, for any $\psi\in\mathcal{A}$ and $(t_0,x_0)\in(0,T)\times\real^N$ such that $u-\psi$ has a local maximum at $(t_0,x_0)$, then there holds
\begin{equation}\label{visc2}
\left\{\begin{array}{ll}\partial_t\psi(t_0,x_0)\leq\Delta_p\psi(t_0,x_0)-|\nabla\psi(t_0,x_0)|, & \text{ if } \ \nabla\psi(t_0,x_0)\neq0,\\
& \\
\partial_t\psi(t_0,x_0)\leq0, & \text{ if } \ \nabla\psi(t_0,x_0)=0.
\end{array}\right.
\end{equation}

\item[(b)]A lower semicontinuous function $u:(0,T)\times\real^N\mapsto\real$ is a \emph{viscosity supersolution} to \eqref{eq1} if $-u$ is a viscosity subsolution to \eqref{eq1}.

\item[(c)] A continuous function $u:(0,T)\times\real^N\mapsto\real$ is a \emph{viscosity solution} to \eqref{eq1} in $(0,T)\times\real^N$ when it is at the same time a viscosity subsolution and a viscosity supersolution.
\end{itemize}
\end{definition}
An immediate consequence of Definition~\ref{def.visc} is that special attention shall be paid to critical points (with respect to the space variable) of subsolutions and supersolutions, this fact being obviously related to the singular behavior of the $p$-Laplacian operator at critical points of $u$ when $p\in (1,2)$. The main abstract results concerning viscosity subsolutions and supersolutions are contained in \cite{OS}. More precisely, the comparison principle is stated in \cite[Theorem~3.9]{OS} and the stability property with respect to uniform limits is \cite[Theorem~6.1]{OS}, both of them being valid in a more general setting encompassing Eq.~\eqref{eq1}. As we shall see below in Lemma~\ref{lem.sing}, this specific notion of viscosity subsolution and supersolutions requires some care to be properly handled.

\subsection{Construction of a viscosity supersolution}

We devote this subsection to the construction of a viscosity supersolution to \eqref{eq1}, in the sense of Definition~\ref{def.visc}. It requires a different analysis at points where the spatial gradient of the supersolution vanishes. As in \cite{BLSS02} for $p=2$ and $q\in (0,1)$ and in \cite{ILS17} for $p\in(p_c,2)$ and $q\in(0,p-1]$, we look for a supersolution in self-similar form.

\begin{lemma}\label{lem.super}
Assume that $p$ and $q$ satisfy \eqref{exp}. There are $\bar{a}>0$ and $\bar{b}>0$ such that, for any $(a,b)\in (\bar{a},\infty)\times (\bar{b},\infty)$, the function
\begin{subequations}\label{super}
\begin{align}
W(t,x) & = (T-t)^{\alpha}f(|x|(T-t)^{\beta}), \qquad (t,x)\in (0,T)\times\real^N, \label{supera} \\
f(y) & =(a+by^{\theta})^{-\gamma}, \qquad y\in [0,\infty), \label{superb}
\end{align}
\end{subequations}
with exponents
\begin{equation}\label{exp.super}
\alpha=\frac{p-q}{p-2q}, \quad \beta=\frac{q-p+1}{p-2q}, \quad \theta=\frac{p}{p-1}, \quad \gamma=\frac{(p-1)q}{p(1-q)}
\end{equation}
is a (classical) supersolution to \eqref{eq1} in $(0,\infty)\times(\real^N\setminus\{0\})$.
\end{lemma}

\begin{proof}
Let $(t,x)\in (0,\infty)\times(\real^N\setminus\{0\})$. We set $y=|x|(T-t)^{\beta}$ and note that
$$
f'(y)=-\gamma b\theta(a+by^{\theta})^{-\gamma-1}y^{\theta-1}
$$
and
$$
f''(y)=-\gamma b\theta(a+by^{\theta})^{-\gamma-1}y^{\theta-2}\left[\theta-1-\theta(\gamma+1)\frac{by^{\theta}}{a+by^{\theta}}\right].
$$
After direct and straightforward (but rather long) calculations we obtain
\begin{equation}\label{expr0}
\begin{split}
\mathcal{L}W(t,x)&:=\partial_tW(t,x)-\Delta_pW(t,x)+|\nabla W(t,x)|^q\\
&=(T-t)^{\alpha-1}\Big[-\alpha f(y)-\beta yf'(y)-(p-1)(|f'|^{p-2}f'')(y) \\
& \hspace{3cm} -\frac{N-1}{y}(|f'|^{p-2}f')(y)+|f'(y)|^q \Big]\\
&=(T-t)^{\alpha-1}(a+by^{\theta})^{-\gamma-1}(H_1(y)+H_2(y)),
\end{split}
\end{equation}
where
\begin{align}
H_1(y)=-\alpha a & +(\gamma b\theta)^{p-1}\left[N-1+(p-1)(\theta-1)-(p-1)\theta(\gamma+1)\frac{by^{\theta}}{a+by^{\theta}}\right] \nonumber\\
& \hspace{1.5cm} \times y^{(\theta-1)(p-1)-1}(a+by^{\theta})^{(\gamma+1)(2-p)} \nonumber \\
=-\alpha a & + (\gamma b\theta)^{p-1}\left[N - p(\gamma+1)\frac{by^{\theta}}{a+by^{\theta}}\right] (a+by^{\theta})^{(\gamma+1)(2-p)},
\label{expr1}
\end{align}
since $(\theta-1)(p-1)=1$ and $(p-1)\theta=p$, and
\begin{equation}\label{expr2}
H_2(y)=(\gamma b\theta)^qy^{q(\theta-1)}(a+by^{\theta})^{(1-q)(\gamma+1)}+(\beta\gamma\theta-\alpha)by^{\theta}.
\end{equation}
Since
$$
\gamma\beta\theta-\alpha=-\frac{1}{1-q}<0
$$
and
$$
q(\theta-1)+\theta(1-q)(\gamma+1)=\theta,
$$
we obtain that
\begin{align}
H_2(y) & \ge (\gamma b\theta)^{q}y^{q(\theta-1)}(by^{\theta})^{(1-q)(\gamma+1)}-\frac{b}{1-q}y^{\theta} \nonumber \\
&=\frac{b}{1-q}y^{\theta}\left[(1-q)(\gamma\theta)^{q}b^{(1-q)\gamma}-1\right] \nonumber \\
& \ge \frac{(\gamma\theta)^q}{2} b^{1+(1-q)\gamma} y^\theta \ge 0, \label{interm9}
\end{align}
provided
\begin{equation}\label{cond3}
b^{(1-q)\gamma}\geq\frac{2}{(1-q)(\gamma\theta)^q}.
\end{equation}

In order to estimate the term $H_1(y)$ we split the range $(0,\infty)$ of $y$ into two regions, one close to the origin and another far from the origin. Let thus $y_0>0$ to be determined later and consider first $y\in(0,y_0]$. Then,
\begin{equation}\label{interm10}
H_1(y)\geq(\gamma b\theta)^{p-1}\left[N-(\gamma+1)p\frac{by_0^{\theta}}{a}\right](a+by^{\theta})^{(2-p)(\gamma+1)}-a\alpha.
\end{equation}
If we require $a>0$, $b>0$, and $y_0>0$ to satisfy
\begin{equation}\label{cond1}
\frac{Na}{2p(\gamma+1)}\geq by_0^{\theta},
\end{equation}
then we infer from \eqref{interm10} that
\begin{align}
H_1(y)&\ge (\gamma b\theta)^{p-1}\frac{N}{2}(a+by^{\theta})^{(2-p)(\gamma+1)}-a\alpha\nonumber \\
&\geq\frac{N(\gamma\theta)^{p-1}}{2}b^{p-1}a^{(2-p)(\gamma+1)}-a\alpha \nonumber\\
&\geq\alpha a^{(2-p)(\gamma+1)}\left[\frac{N(\gamma\theta)^{p-1}}{2\alpha}b^{p-1}-a^{1-(2-p)(\gamma+1)}\right]\geq0, \label{PhZ1}
\end{align}
provided that
\begin{equation}\label{cond2}
\frac{N(\gamma\theta)^{p-1}}{2\alpha}b^{p-1}\geq a^{1-(2-p)(\gamma+1)}.
\end{equation}

We turn now our attention to the complementary region $y>y_0$. We use the obvious bound $by^{\theta}/(a+by^{\theta})<1$ to find
$$
H_1(y)\geq-(\gamma b\theta)^{p-1}p(\gamma+1)(a+by^{\theta})^{(2-p)(\gamma+1)}-a\alpha,
$$
hence, putting $L:=p(\gamma+1)(\gamma\theta)^{p-1}$, we further deduce from \eqref{interm9} that
\begin{equation*}
\begin{split}
(H_1&+H_2)(y)\geq\frac{(\gamma\theta)^q}{2}b^{1+(1-q)\gamma}y^{\theta}-a\alpha-Lb^{p-1}(a+by^{\theta})^{(2-p)(\gamma+1)}\\
&\geq\frac{(\gamma\theta)^q}{4}b^{(1-q)\gamma}by_0^{\theta}-a\alpha\\
&+\frac{(\gamma\theta)^q}{4}b^{1+(1-q)\gamma}y^{\theta}-Lb^{p-1+(2-p)(\gamma+1)}y^{\theta-(p-2q)/(1-q)}\left[1+\frac{a}{by_0^{\theta}}\right]^{(2-p)(\gamma+1)}\\
&\geq b^{1+(1-q)\gamma}y^{\theta-(p-2q)/(1-q)}\left[\frac{(\gamma\theta)^q}{4}y^{(p-2q)/(1-q)}-L\left(1+\frac{a}{by_0^{\theta}}\right)^{(2-p)(\gamma+1)}b^{(q-p+1)\gamma}\right],
\end{split}
\end{equation*}
provided that
\begin{equation}\label{cond4}
\frac{(\gamma\theta)^{q}}{4\alpha}b^{(1-q)\gamma+1}y_0^{\theta}\geq a.
\end{equation}
We now choose
\begin{equation}
a=\lambda by_0^{\theta}, \label{PhZ2}
\end{equation}
with $\lambda>0$ to be specified later. Then
\begin{align}
(H_1+H_2)(y)&\geq b^{1+(1-q)\gamma}y^{\theta-(p-2q)/(1-q)}\Big[ \frac{(\gamma\theta)^q}{4}y_0^{(p-2q)/(1-q)} \nonumber\\
& \hspace{5cm} - L(1+\lambda)^{(2-p)(\gamma+1)}b^{(q-p+1)\gamma}\Big] \geq 0, \label{PhZ3}
\end{align}
provided
\begin{equation}\label{cond5}
y_0^{(p-2q)/(1-q)}\geq\frac{4L(1+\lambda)^{(2-p)(\gamma+1)}}{(\gamma\theta)^q}b^{(q-p+1)\gamma}.
\end{equation}

Summarizing, according to \eqref{interm9}, \eqref{PhZ1}, and \eqref{PhZ3}, we have established that $(H_1+H_2)(y)\ge 0$ for all $y\in (0,\infty)$ as soon as the parameters $a$, $b$, $y_0$, and $\lambda$ satisfy \eqref{cond3}, \eqref{cond1}, \eqref{cond2}, \eqref{cond4}, and \eqref{PhZ2}. We end the proof by choosing the parameters $b$, $\lambda$, and $y_0$ in order to ensure the compatibility of all the conditions we had to impose along the way in the estimates. First of all, we set
$$
\lambda = \frac{2p(\gamma+1)}{N},
$$
which implies the validity of \eqref{cond1}. Moreover, from \eqref{cond3} and \eqref{cond4} we have to choose $b>0$ such that
\begin{equation}\label{cond7}
b^{(1-q)\gamma}\geq\max\left\{\frac{2}{(1-q)(\gamma\theta)^q},\frac{4\lambda\alpha}{(\gamma\theta)^q}\right\}.
\end{equation}
Finally, inserting \eqref{PhZ2} into \eqref{cond2}, we readily deduce that
\begin{equation}\label{cond6}
\frac{N(\gamma\theta)^{p-1}\lambda^{\gamma(2-p)-p+1}}{2\alpha}b^{\gamma(2-p)}\geq y_0^{(p-2q)/(1-q)}.
\end{equation}
Let us notice that, since $2-p>q-p+1$, the conditions \eqref{cond5}, \eqref{cond7}, and \eqref{cond6} can be met simultaneously by choosing $b>0$ sufficiently large, which ends the proof.
\end{proof}

Now, let $T>0$, $a>\bar{a}$, and $b>\bar{b}$, and consider the function $W$ defined by \eqref{super}. With the aim of showing that $W$ is a viscosity supersolution to \eqref{eq1} in the sense of Definition~\ref{def.visc}, let $\psi\in\mathcal{A}$ and $(t_0,x_0)\in (0,T)\times\mathbb{R}^N$ be such that $W-\psi$ has a local minimum at $(t_0,x_0)$. Since both $W$ and $\psi$ belong to $C^1([0,T]\times\mathbb{R}^N)$, this property implies that
\begin{equation}
\partial_t W(t_0,x_0) = \partial_t \psi(t_0,x_0) \;\text{ and }\; \nabla W(t_0,x_0) = \nabla \psi(t_0,x_0). \label{PhZ4}
\end{equation}
Since $\nabla W(t_0,x_0)\ne 0$ when $x_0\ne 0$, Lemma~\ref{lem.super} and \eqref{PhZ4} guarantee that the condition to be a viscosity supersolution is fulfilled if $x_0\ne 0$. No information is provided by Lemma~\ref{lem.super} if $x_0=0$. In that case, we might actually face a problem. Indeed, for $W$ to meet the requirement of viscosity solutions when $W-\psi$ has a local minimum at $(t_0,0)$ for some $t_0\in (0,T)$, the inequality $\partial_t\psi(t_0,0)\ge 0$ has to be satisfied according to Definition~\ref{def.visc}. However, recalling \eqref{PhZ4}, we realize that
$$
\partial_t\psi(t_0,0)=\partial_tW(t_0,0)=-\alpha(T-t_0)^{\alpha-1}a^{-\gamma}<0,
$$
and an apparent contradiction. This is in fact an artificial problem: there \emph{do not exist} any admissible function $\psi$ such that $W-\psi$ attains a local minimum at a point $(t_0,0)$ as the following lemma shows.

\begin{lemma}\label{lem.sing}
Let $T>0$, $a>0$, and $b>0$ and consider the function $W$ defined by \eqref{super}, the exponents $p$ and $q$ still satisfying \eqref{exp}. Let $\psi\in \mathcal{A}$ and assume that $(t_0,x_0)\in(0,T)\times\real^N$ is a local minimum for $W-\psi$. Then $x_0\neq0$.
\end{lemma}

\begin{proof}
Assume for contradiction that $x_0=0$. On the one hand, since $W\in C^1([0,T]\times\real^N)$, we have $\nabla\psi(t_0,0)=\nabla W(t_0,0)=0$. On the other hand, $\psi\in\mathcal{A}$ and there exist a function $\xi\in\mathcal{F}_p$, a modulus of continuity $\omega\in C([0,\infty))$, $\omega\ge 0$ and a sufficiently small $\delta>0$ such that, for $(t,x)\in (t_0-\delta,t_0+\delta)\times B_\delta(0)$,
\begin{equation}\label{interm30}
|\psi(t,x)-\psi(t_0,0)-\partial_t\psi(t_0,0)(t-t_0)|\leq\xi(|x|)+\omega(|t-t_0|).
\end{equation}
In particular for $t=t_0$, \eqref{interm30} becomes
$$
|\psi(t_0,x)-\psi(t_0,0)|\leq\xi(|x|), \quad \text{ for } \ x\in B_\delta(0).
$$
Furthermore, since $(t_0,0)$ is a local minimum of $W-\psi$, we realize that
\begin{equation}\label{interm31}
W(t_0,0)-W(t_0,x)\leq\psi(t_0,0)-\psi(t_0,x)\leq\xi(|x|), \quad \text{ for } \ x\in B_\delta(0).
\end{equation}
Taking into account the formula \eqref{super} for $W$, we infer from \eqref{interm31} that
$$
(T-t_0)^{\alpha}\left[a^{-\gamma}-(a+b|x|^{\theta}(T-t_0)^{\theta\beta})^{-\gamma}\right]\leq\xi(|x|), \quad \text{ for } \ x\in B_\delta(0),
$$
hence, as $|x|\to 0$,
\begin{equation*}
\frac{(T-t_0)^{\alpha}}{(a+b|x|^{\theta}(T-t_0)^{\theta\beta})^{\gamma}}\left[\frac{b\gamma}{a}(T-t_0)^{\theta\beta}|x|^{\theta}+o(|x|^{\theta})\right]\leq\xi(|x|).
\end{equation*}
Consequently, recalling that $\theta=p/(p-1)$,
\begin{equation}\label{interm32}
0<\frac{b \gamma (T-t_0)^{\alpha+\theta\beta}}{a^{\gamma+1}}\leq\liminf\limits_{r\to 0} \frac{\xi(r)}{r^{p/(p-1)}}.
\end{equation}
Next, since $\xi\in\mathcal{F}_p$, we infer from \eqref{PhF2} that
$$
\lim\limits_{r\to 0} \frac{\xi'(r)}{r^{1/(p-1)}} = 0,
$$
and a further application of l'Hospital's rule gives
$$
\lim\limits_{r\to 0} \frac{\xi(r)}{r^{p/(p-1)}} = 0,
$$
thereby contradicting \eqref{interm32}. Therefore, we cannot have $x_0=0$, ending the proof.
\end{proof}
Combining Lemma~\ref{lem.super} and Lemma~\ref{lem.sing}, we infer from the discussion preceding the statement of Lemma~\ref{lem.sing} that, for $T>0$, $a>\bar{a}$, and $b>\bar{b}$, the function $W$ defined in \eqref{super} is a viscosity supersolution to \eqref{eq1} in the whole $(0,T)\times\real^N$ in the sense of Definition~\ref{def.visc}. Summarizing, we have established the following result.
\begin{corollary}\label{cor.sup}
Assume that $p$ and $q$ satisfy \eqref{exp}. For $T>0$, $a>\bar{a}$, and $b>\bar{b}$, the function $W$ defined in \eqref{super} is a viscosity supersolution to \eqref{eq1} in $(0,T)\times\real^N$.
\end{corollary}

A similar construction (already performed in \cite{ILS17}) gives us a subsolution to \eqref{eq1} that will be used for comparison from below in order to show positivity and non-extinction when $u_0$ satisfies \eqref{tail.slow}. We recall it here for the sake of completeness.
\begin{lemma}\label{lem.sub}
Assume that $p$ and $q$ satisfy \eqref{exp}. There exists $b_0>0$ depending only on $p$ and $q$ such that, given $T>0$ and $b\in(0,b_0)$, there is $A(b,T)>0$ depending only on $N$, $p$, $q$, $b$, and $T$ such that the function
$$
w(t,x):=(T-t)^{1/(1-q)}(a+b|x|^{\theta})^{-\gamma}, \qquad \theta=\frac{p}{p-1}, \ \gamma=\frac{q(p-1)}{(1-q)p},
$$
is a subsolution to \eqref{eq1} in $(0,T)\times\real^N$ provided $a>A(b,T)$.
\end{lemma}
\begin{proof}
The proof is totally identical to that of \cite[Lemma~6.1]{ILS17}. In fact, in the quoted reference, it is assumed that $0<q<p-1$, but a simple inspection of the proof shows that it works identically for any $q\in(0,p/2)$.
\end{proof}

\subsection{Proof of Theorem~\ref{th.tail}}

With these constructions, we are now in a position to prove the optimality of the spatial decay \eqref{tail.optimal} for finite time extinction to take place.

\begin{proof}[Proof of Theorem~\ref{th.tail}]
\noindent \textbf{Extinction with optimal tail.} Let $u$ be a solution to the Cauchy problem \eqref{eq1}-\eqref{IC} with an initial condition $u_0$ satisfying \eqref{tail.optimal} and consider $a>\bar{a}$ and $b>\bar{b}$. For $T>0$, it follows from Corollary~\ref{cor.sup} that
$$
W(t,x)=(T-t)^{\alpha}(a+b(T-t)^{\beta\theta}|x|^{\theta})^{-\gamma}, \qquad (t,x)\in (0,T)\times\mathbb{R}^N,
$$
is a supersolution to \eqref{eq1} in $(0,T)\times\mathbb{R}^N$, the parameters $\theta$ and $\gamma$ being given as usual by
$$
\theta=\frac{p}{p-1}>1, \qquad \gamma=\frac{q(p-1)}{(1-q)p}.
$$
For $x\in\mathbb{R}^N$,
\begin{equation*}
\begin{split}
W(0,x)&=T^{\alpha}(a+bT^{\beta\theta}|x|^{\theta})^{-\gamma}=T^{\alpha-\beta\theta\gamma}b^{-\gamma}\left[\frac{a}{bT^{\beta\theta}}+|x|^{\theta}\right]^{-\gamma}\\
&=T^{1/(1-q)}b^{-\gamma}\left[\frac{a}{bT^{\beta\theta}}+|x|^{\theta}\right]^{-\gamma},
\end{split}
\end{equation*}
since $\alpha - \beta \theta \gamma = 1/(1-q)$. Choose in a first step $T>0$ sufficiently large such that
$$
\frac{a}{bT^{\beta\theta}}<1.
$$
Then, taking into account that $\theta>1$ and the elementary inequality $1+|x|^{\theta}\leq(1+|x|)^{\theta}$ for any $x\in\real^N$, we further infer from \eqref{tail.optimal} that
\begin{equation*}
\begin{split}
W(0,x)&\geq T^{1/(1-q)}b^{-\gamma}(1+|x|^{\theta})^{-\gamma}\geq T^{1/(1-q)}b^{-\gamma}(1+|x|)^{-q/(1-q)}\\
&\geq\frac{T^{1/(1-q)}b^{-\gamma}}{C_0}C_0(1+|x|)^{-q/(1-q)}\geq\frac{T^{1/(1-q)}b^{-\gamma}}{C_0}u_0(x)\geq u_0(x),
\end{split}
\end{equation*}
provided we take $T>0$ sufficiently large such that
$$
T^{1/(1-q)}b^{-\gamma}\geq C_0.
$$
Thus, for $T$ sufficiently large, we deduce from the comparison principle that
$$
u(t,x)\leq W(t,x), \qquad (t,x)\in(0,T)\times\real^N,
$$
and it is immediate to conclude that this implies extinction in finite time for $u$, with an extinction time $T_e\leq T$.

\medskip

\noindent \textbf{Non-extinction with slower tail.} Let us now consider a solution $u$ to the Cauchy problem \eqref{eq1}-\eqref{IC} with an initial condition $u_0$ satisfying \eqref{tail.slow}. Then the non-extinction in finite time and the positivity for any $t>0$ (that is, $\mathcal{P}(t)=\real^N$ for any $t>0$) follow from the same proof as in \cite[Section 6]{ILS17}, which applies identically also for the range $q\in[p-1,p/2)$. Thus, optimality of the tail in \eqref{tail.optimal} is proved.
\end{proof}

\section{Optimal extinction rates}\label{sec.rate}

This section is devoted to the proof of Theorem~\ref{th.rate}. We thus assume from now on that the exponents $p$ and $q$ satisfy \eqref{exp} as well as $q>p-1$. Assume also that $u_0$ satisfies \eqref{regIC} and \eqref{tail.fast1} for some constant $K_0>0$. Throughout this section, $C$ and $C_i$, $i\ge 1$, denote positive constants depending only on $N$, $p$, $q$, and $u_0$. Dependence upon additional parameters shall be indicated explicitly.

We begin with the proof of the lower bound, which relies on the derivation of a functional inequality for the $L^\infty$-norm of $u$. Exploiting this functional inequality requires the following preliminary result.

\begin{lemma}\label{lem.fi}
Let $T>0$ and a function $h:[0,T]\to [0,\infty)$ such that
\begin{equation}
\mu(t) := \inf_{s\in [0,t]}\{ h(s) \} > 0, \qquad t\in (0,T), \qquad h(T)=0\ , \label{PhZ5}
\end{equation}
and
\begin{equation}
\delta (t-s) h(t)^m \le h(s)\ , \qquad 0<s<t<T\ , \label{PhZ6}
\end{equation}
for some $m\in (0,1)$ and $\delta>0$. Then
\begin{equation}
h(t) \ge \left( \frac{\delta^{1-m}}{2} \right)^{1/(1-m)^2} (T-t)^{1/(1-m)}, \qquad t\in [0,T]. \label{PhZ7}
\end{equation}
\end{lemma}

\begin{proof}
Fix $t\in (0,T)$ and $\tau\in (t,T)$. Introducing the sequence $(t_i)_{i\ge 0}$ defined by
$$
t_i := \frac{t}{2^i} + \left( 1 - \frac{1}{2^i} \right) \tau, \qquad i\ge 0,
$$
we observe that
\begin{equation}
t=t_0 < t_i < t_{i+1} < \tau, \qquad i\ge 1\ , \qquad \lim\limits_{i\to\infty} t_i = \tau. \label{PhZ10}
\end{equation}
Since $t_{i+1}-t_i = (\tau-t)/2^{i+1}$ for $i\ge 0$, we infer from \eqref{PhZ6} that
\begin{equation}
\frac{\delta(\tau-t)}{2^{i+1}} h(t_{i+1})^m \le h(t_i), \qquad i\ge 0. \label{PhZ8}
\end{equation}
By an induction argument, we deduce from \eqref{PhZ8} that
\begin{equation}
h(t)=h(t_0) \ge \frac{(\delta(\tau-t))^{\Sigma_i}}{2^{\sigma_i}} h(t_i)^{m^i}, \qquad i\ge 1, \label{PhZ9}
\end{equation}
where
$$
\Sigma_i := \sum_{j=0}^{i-1} m^j = \frac{1-m^i}{1-m} \ \text{ and }\ \sigma_i := \sum_{j=0}^{i-1} (j+1)m^j = \frac{1-(i+1)m^i + i m^{i+1}}{(1-m)^2}
$$
for $i\ge 1$. We then infer from \eqref{PhZ5}, \eqref{PhZ10}, and \eqref{PhZ9} that
$$
h(t) \ge \frac{(\delta(\tau-t))^{\Sigma_i}}{2^{\sigma_i}} \mu(\tau)^{m^i}, \qquad i\ge 1.
$$
Owing to the positivity of $\mu(\tau)$, we may pass to the limit as $i\to\infty$ in the previous inequality to obtain
$$
h(t) \ge \left( \frac{(\delta(\tau-t))^{1-m}}{2} \right)^{1/(1-m)^2}.
$$
We then let $\tau\to T$ in the previous inequality to complete the proof.
\end{proof}

\medskip

\begin{proof}[Proof of Theorem~\ref{th.rate}: $L^\infty$-lower bound.]
By \cite[Lemma~5.1]{IL12}, there exists $C_1>0$ such that
\begin{equation}\label{L1Linfty}
\|u(t)\|_1\leq C_1\|u(t)\|_{\infty}^{\nu},
\end{equation}
with
$$
\nu:=\frac{(N+1)(q_*-q)}{p-q}, \quad q_*:=p-\frac{N}{N+1}.
$$
Also by \cite[Theorem~1.7]{IL12} we have the gradient estimate
\begin{equation}\label{grad.est1}
\|\nabla u(t)\|_{\infty}\leq C_2\|u(s)\|_{\infty}^{1/q}(t-s)^{-1/q}, \qquad {\rm for \ any} \ 0<s<t.
\end{equation}
Let $t>0$ and $s\in(0,t)$. We infer from the Gagliardo-Nirenberg inequality and the estimates \eqref{L1Linfty} and \eqref{grad.est1} that
\begin{equation*}
\begin{split}
\|u(t)\|_{\infty}&\leq C\|\nabla u(t)\|_{\infty}^{N/(N+1)}\|u(t)\|_1^{1/(N+1)}\\
&\leq C C_1^{1/(N+1)}C_2^{N/(N+1)} \|u(s)\|_{\infty}^{N/q(N+1)}\|u(t)\|_{\infty}^{\nu/(N+1)}(t-s)^{-N/q(N+1)},
\end{split}
\end{equation*}
from which, taking into account that
$$
1-\frac{\nu}{N+1}=1-\frac{q_*-q}{p-q}=\frac{N}{(N+1)(p-q)},
$$
we derive that
\begin{equation}\label{interm1}
(t-s)\|u(t)\|_{\infty}^{q/(p-q)}\leq C_3\|u(s)\|_{\infty}.
\end{equation}
Let $T_e$ be the extinction time of $u$. Since $q<p-q$, it follows from the properties of $u$ prior to the extinction time that we are in a position to apply Lemma~\ref{lem.fi} (with $h=\|u\|_\infty$, $T=T_e$, and $m=q/(p-q)<1$) and obtain the claimed lower bound.
\end{proof}

\medskip

\begin{proof}[Proof of Theorem~\ref{th.rate}: upper bounds.] We start again from results contained in \cite{IL12}. More precisely, it follows from \cite[Eq.~(5.5)]{IL12} and \eqref{tail.fast1} that
\begin{equation}\label{tail.fast2}
0\leq u(t,x)\leq C_4 |x|^{-(p-q)/(q-p+1)}, \qquad (t,x)\in(0,\infty)\times\real^N.
\end{equation}
Moreover, we have the following gradient estimate \cite[Theorem~1.3(iii)]{IL12}
$$
\left|\nabla u^{-(q-p+1)/(p-q)}(t,x)\right|\leq C\left[1+\|u_0\|_{\infty}^{(p-2q)/p(p-q)}t^{-1/p}\right]
$$
for $(t,x)\in[0,T_e)\times\real^N$. Restricting ourselves to $t\in(T_e/2,T_e)$, the right hand side of the previous inequality is bounded and we further obtain
\begin{equation}\label{grad.est2}
|\nabla u(t,x)|\leq C_5u(t,x)^{1/(p-q)}, \qquad (t,x)\in(T_e/2,T_e)\times\real^N.
\end{equation}
Let $t\in (T_e/2,T_e)$. Integrating \eqref{eq1} over $(t,T_e)\times \real^N$ and using \eqref{grad.est2} as well as the property $\|u(T_e)\|_1=0$, we find
\begin{equation}\label{interm4}
\|u(t)\|_1=\int_t^{T_e}\int_{\real^N}|\nabla u(s,x)|^q\,dx\,ds\leq C_5^q\int_t^{T_e}\int_{\real^N}|u(s,x)|^{q/(p-q)}\,dx\,ds.
\end{equation}
Since $p>2q$, we have $q/(p-q)\in(0,1)$ and we infer from \eqref{tail.fast2} and H\"older's inequality that, for any $R\in(0,\infty)$ and $s\in (t,T_e)$,
\begin{equation*}
\begin{split}
\int_{\real^N}|u(s,x)|^{q/(p-q)}\,dx&=\int_{B_R(0)}|u(s,x)|^{q/(p-q)}\,dx+\int_{\real^N\setminus B_R(0)}|u(s,x)|^{q/(p-q)}\,dx\\
&\leq\left[\int_{B_R(0)}u(s,x)\,dx\right]^{q/(p-q)}\left(\int_{B_R(0)}dx\right)^{(p-2q)/(p-q)}\\
&\quad +C\int_{R}^{\infty}r^{N-1-q/(q-p+1)}\,dr\\
&\leq C\left[\|u(s)\|_{1}^{q/(p-q)}R^{N(p-2q)/(p-q)}+R^{N-q/(q-p+1)}\right],
\end{split}
\end{equation*}
where, in order to derive the last inequality, we took into account that, since $p-1<q<p/2$ and $p>p_c$,
\begin{align*}
N-\frac{q}{q-p+1} & =\frac{(N-1)q-N(p-1)}{q-p+1}<\frac{(N-1)p-2N(p-1)}{2(q-p+1)}\\
& =\frac{(N+1)(p_c-p)}{2(q-p+1)}<0.
\end{align*}
We next optimize in $R$ with the choice
$$
\|u(s)\|_1^{q/(p-q)}R^{N-Nq/(p-q)}=R^{N-q/(q-p+1)},
$$
or equivalently
$$
R=\|u(s)\|_1^{-(q-p+1)/[(N+1)(p-q)-N]}.
$$
Substituting this choice of $R$ in the previous inequality leads us to
\begin{equation}\label{interm5}
\int_{\real^N}|u(s,x)|^{q/(p-q)}\,dx\leq C\|u(s)\|_1^{\omega},
\end{equation}
with
$$
\omega:=\frac{q}{p-q}-\frac{N(p-2q)(q-p+1)}{(p-q)[(N+1)(p-q)-N]}.
$$
We observe after straightforward calculations that, since $p-1<q<p/2<q_*=p-N/(N+1)$, there holds
$$
1-\omega=\frac{p-2q}{(N+1)(p-q)-N}>0.
$$
Now, combining \eqref{interm4} and \eqref{interm5} gives
\begin{equation}\label{interm6}
\|u(t)\|_1\leq C_6\int_t^{T_e}\|u(s)\|_1^{\omega}\,ds, \qquad t\in (T_e/2,T_e).
\end{equation}
It readily follows from \eqref{eq1} and the non-negativity of $u$ that $s\mapsto \|u(s)\|_1$ is non-increasing and we infer from \eqref{interm6} that
$$
\|u(t)\|_1\leq C_6 (T_e-t) \|u(t)\|_1^{\omega}, \qquad t\in (T_e/2,T_e).
$$
Therefore, since $\|u(t)\|_1\ne 0$ for $t\in (T_e/2,T_e)$,
\begin{equation}\label{interm7}
\|u(t)\|_1\leq C_7(T_e-t)^{[(N+1)(p-q)-N]/(p-2q)}, \qquad t\in(T_e/2,T_e),
\end{equation}
and we have established the upper bound in \eqref{rate.optimal.L1}. It next follows from \eqref{grad.est2} and the Gagliardo-Nirenberg inequality that, for $t\in(T_e/2,T_e)$,
\begin{align*}
\|u(t)\|_{\infty} & \le C\|\nabla u(t)\|_{\infty}^{N/(N+1)}\|u(t)\|_1^{1/(N+1)} \\
& \leq C\|u(t)\|_{\infty}^{N/(N+1)(p-q)}\|u(t)\|_1^{1/(N+1)},
\end{align*}
hence
\begin{equation}\label{interm8}
\|u(t)\|_{\infty}\leq C\|u(t)\|_1^{(p-q)/[(N+1)(p-q)-N]}.
\end{equation}
Gathering \eqref{interm7} and \eqref{interm8}, we readily obtain the upper bound in \eqref{rate.optimal}, as desired.
\end{proof}

\begin{proof}[Proof of Theorem~\ref{th.rate}: $L^1$-lower bound.]
We are left with proving the $L^1$-lower bound for $t\in (T_e/2,T_e)$. To this end, we use once more the Gagliardo-Nirenberg inequality along with \eqref{grad.est2} to obtain,
$$
\|u(t)\|_\infty \le C \|\nabla u(t)\|_\infty^{N/(N+1)} \|u(t)\|_1^{1/(N+1)} \le C \|u(t)\|_\infty^{N/(p-q)(N+1)} \|u(t)\|_1^{1/(N+1)}.
$$
Since $\|u(t)\|_\infty\ne 0$, we further obtain
$$
\|u(t)\|_\infty^{[(N+1)(p-q)-N]/(p-q)} \le \|u(t)\|_1,
$$
from which the lower bound in \eqref{rate.optimal.L1} readily follows with the help of the lower bound in \eqref{rate.optimal}.
\end{proof}

\section*{Acknowledgements} The first author is partially supported by
the ERC Starting Grant GEOFLUIDS 633152. Part of this
work was done while the first author enjoyed the hospitality and
support of the Institute de Math\'ematiques de Toulouse, Toulouse,
France.

\bibliography{OptExtRatesPLEAbs}{}
\bibliographystyle{siam}

\end{document}